\documentclass[reqno]{amsart}
\usepackage{amsmath,amsthm,amscd,amssymb,amsfonts, amsbsy}
\usepackage{latexsym, color, enumerate}
\usepackage[mathscr]{eucal}

\usepackage[hmargin=3.20 cm,vmargin=3.20 cm]{geometry}

\theoremstyle{plain}
\newtheorem{theorem}[equation]{Theorem}
\newtheorem{lemma}[equation]{Lemma}
\newtheorem{corollary}[equation]{Corollary}

\theoremstyle{definition}

\theoremstyle{remark}
\newtheorem{remark}[equation]{Remark}

\numberwithin{equation}{section}

\newcommand{\bP}{\mathbb{P}}

\newcommand{\bR}{\mathbb{R}}

\newcommand\cX{\mathcal{X}}
\newcommand\cY{\mathcal{Y}}

\providecommand{\norm}[1]{\lVert#1\rVert}

\newcommand{\A}{\mathcal{A}}

\newcommand{\G}{\mathbf{G}}

\begin{document}

 \title[Navier-Stokes equations in critical mixed-norm Lebesgue spaces]{Well-posedness for the Navier-Stokes equations in  critical mixed-norm Lebesgue spaces}

\author[T. Phan]{Tuoc Phan}
\address[T. Phan]{Department of Mathematics, University of Tennessee, 227 Ayres Hall,
1403 Circle Drive, Knoxville, TN 37996-1320 }
\email{phan@math.utk.edu}

\thanks{T. Phan's research is partially supported by the Simons Foundation, grant \# 354889.}

\begin{abstract}  We study the Cauchy problem in $n$-dimensional space for the system of Navier-Stokes equations in critical mixed-norm Lebesgue spaces. Local well-posedness and global well-posedness of solutions are established in the class of critical mixed-norm Lebesgue spaces. Being in the mixed-norm Lebesgue spaces, both of the initial data and the class of solutions could be singular at certain points or decaying to zero at infinity with different rates in different spatial variable directions.  Some of these singular rates could be very strong and some of the decaying rates could be significantly slow.  Besides other interests, the results of the paper particularly show an interesting phenomena on the persistence of the anisotropic behavior of the initial data under the evolution. To achieve the goals, fundamental analysis theory such as Young's inequality, time decaying of solutions for heat equations, the boundedness of the Helmholtz-Leray projection, and   the boundedness of  the Riesz tranfroms are developed in mixed-norm Lebesgue spaces. These fundamental analysis results are independently topics of  great interests and they are potentially useful in other problems.
\end{abstract}

\maketitle

\today
\section{Introduction and main results}

This paper establishes local and global well-posedness of the Cauchy problem for Navier-Stokes equations in critical mixed norm Lebesgue spaces.  
We consider the following initial value problem for the system of Navier-Stokes equations of incompressible fluid in $n$-dimensional space
\begin{equation} \label{NS.eqn}
\left\{
\begin{array}{cccl}
u_t -\Delta u + (u \cdot \nabla) u + \nabla P & = & 0 & \quad \text{in} \quad \bR^n \times (0, T), \\
\text{div}(u) & =& 0 & \quad \text{in} \quad \bR^n  \times (0, T), \\
u_{| t =0}& =& a_0 & \quad \text{on} \quad  \bR^n,
\end{array} \right.
\end{equation}
where $u =  (u_1, u_2, \cdots, u_n): \bR^n \times (0, T) \rightarrow \bR^n$ is the unknown velocity of the considered fluid with some $T>0$ and $n\geq 2$. Moreover,  $P : \bR^n \times (0, T) \rightarrow  \bR$ is the unknown fluid pressure, and $a_0$ is a given vector field initial data function which is assumed to be divergence-free.   Global well-posedness of small solutions in critical mixed-norm Lebesgue spaces and local well-posenedness for large solutions in critical mixed-norm Lebesgue spaces are established. Being in the mixed-norm Lebesgue spaces, both of the initial data and the solutions obtained in the paper could possibly decay to zero with different rates as $|x| \rightarrow \infty$ in different directions. Similarly, they could also be singular at certain points in $\bR^n$ with different rates in different directions of the  spatial $x$-variable.  As a result, this paper demonstrates an important phenomenon on the persistence of the anisotropic properties of the initial data under the evolution of  the Navier-Stokes equations.

To explain the ideas, motivation and to put our results in perspective, let us review and discuss known results concerning the Cauchy problem for the system of the Navier-Stokes equations \eqref{NS.eqn} with possibly irregular initial data in critical spaces. In 1984,  in the well-known work \cite{Kato}, T. Kato initiated the study of \eqref{NS.eqn} with initial data belonging to the space $L_n(\bR^n)$ and he proved the global existence and uniqueness of solutions of \eqref{NS.eqn} in a subspace of $C([0,\infty), L_n(\bR^n))$  provided the norm $\norm{a_0}_{L_n(\bR^n)}$
is sufficiently small. Similarly, local existence and uniqueness of solutions were also obtained in \cite{Kato} with initial data $a_0 \in L_n(\bR^n)$.  As found in \cite{Giga-M,  Kato-1, K-Ya, Kazono-Ya, Taylor}, in  \cite{Cannone-2, Cannone-Planchon, F-Planchon} and \cite[Theorem 5.40, p. 234]{Chemin},  this kind of global and local existence and uniqueness of solutions  continues to hold with initial data in homogeneous Morrey spaces $\mathcal{M}^{q,q}(\bR^n)$ for $1\leq q \leq n$, and respectively in homogeneous Besov spaces $\dot{B}_{p_0, \infty}^{-1 +\frac{n}{p_0}}(\bR^n)$ for $n \leq p_0 < \infty$.  Here,  for $1\leq q <\infty$ and $0< \lambda \leq n$, we say that a $L_q$-locally integrable function $f: \bR^n \rightarrow \bR$ belongs to the Morrey space $\mathcal{M}^{q, \lambda}(\bR^n)$ provided that its norm
\begin{equation*} 
\norm{f}_{\mathcal{M}^{q, \lambda}(\bR^n)} = \sup_{B_\rho(x_0)\subset \bR^n} \left\{ \rho^{\lambda-n} \int_{B_\rho(x_0)} |f(x)|^q dx\right\}^{\frac{1}{q}}  <\infty,
\end{equation*}
where $B_\rho(x_0)$ denotes the ball in $\bR^n$ of radius $\rho>0$ and centered at $x_0 \in \bR^n$. Also, for $q \in [1,\infty]$ and $\alpha >0$, $\dot{B}_{q, \infty}^{-\alpha}(\bR^n)$ denotes the homogeneous Besov space consisting of distributions $f$  whose norm can be equivalently defined by 
\begin{equation} \label{Besov-norm}
\norm{f}_{\dot B^{-\alpha}_{q, \infty}(\bR^n)}\approx \sup_{t>0}  t^{\frac{\alpha}{2}} \norm{e^{\Delta t} f(\cdot)}_{L_q(\bR^n)} < \infty.
\end{equation}

The significant breakthrough is due to the work \cite{Koch-Tataru} by H. Koch and D. Tataru in 2001. In this work, the authors established the global well-posedness of the Cauchy problem \eqref{NS.eqn} for small initial data in the borderline $\text{BMO}^{-1}(\bR^n)$ space. Here, the space $\text{BMO}^{-1}(\bR^n)$ can be defined as the space of all distributional divergences of $\text{BMO}(\bR^n)$ vector fields.  
On the other hand, it should be also noted that it has been shown recently by J. Bourgain and N. Pavlovi\'{c} in \cite{Bourgain-Pavlovic} that the Cauchy problem \eqref{NS.eqn} is ill-posedness in a space even smaller than $\dot B^{-1}_{\infty, \infty}(\bR^n)$.

Now, we would like to note that all of the spaces appear in the mentioned papers are invariant with respect to the scaling 
\begin{equation} \label{scaling}
f (\cdot) \rightarrow \lambda f (\lambda \cdot), \quad  \lambda  > 0
\end{equation}
in the sense that for every $f$  in some space $E$ that we just mentioned, then
\[
\norm{f(\cdot)}_{E}  = \norm{\lambda f(\lambda \cdot)}_{E}, \quad \forall \  \lambda >0.
\]
In other words,  up to now, $\text{BMO}^{-1}(\bR^n)$ is the largest known space that is invariant under the scaling \eqref{scaling} on which the Cauchy problem for the system of the Navier-Stokes equations \eqref{NS.eqn} is globally well-posed for small initial data. Interested readers may find in \cite{Giga-1, Giga} for related results in bounded domains, and in \cite[Chapter 5]{Chemin}, \cite[Chapters 7 - 9]{L-R} and \cite[Chapter 5]{Tsai} for further results, discussion, and more related references.

Motivated by the mentioned work, this paper continues the study of the well-posedness of the Cauchy problem \eqref{NS.eqn} in critical spaces. We plan to refine and extend all the mentioned known work to a completely new and interesting direction. In this paper, we particularly focus on the Lebesgue space setting. Unlike the mentioned results, we investigate the class of initial data and solutions for the Cauchy problem \eqref{NS.eqn} that possibly decay to zero with different rates  as $|x| \rightarrow \infty$ in different directions. Some of these rates could be extremely slow. Similarly, the class of initial data and solutions investigated in this paper could also be singular at certain points in $\bR^n$ with different singularity rates in different spatial directions, and some of which could be very strong.  As the initial data and the solutions are in the same class of such functions, and besides other interests, the results of this paper particularly demonstrate the persistence of the anisotropic properties of the initial data under the evolution of the Navier-Stokes equations.  To the best of our knowledge, this phenomenon is even not known for the heat equation.  To achieve the goals, we follow the spirit of Krylov in the work \cite{Krylov} and use mixed norm Lebesgue spaces to capture the features of those kinds of functions. Several important analysis inequalities and estimates in mixed norm Lebesgue spaces will be also developed in this paper. See also \cite{Dong-Kim, Dong-Krylov, Dong-Phan, TP} for some other related work and \cite{Krylov-survey} for a survey paper on some interesting features regarding mixed norm Lebesgue spaces. 

For $p_1, p_2, \cdots, p_n \in [1, \infty)$, and for a given measurable function $f: \bR^n \rightarrow \bR$, we say that $f$ belongs to the the mixed-norm Lebesgue space $L_{p_1p_1\cdots p_n}(\bR^n)$ if its norm
\[
\begin{split}
& \norm{f}_{L_{p_1p_1\cdots p_n}(\bR^n)} \\
&  = \left( \left( \left(\cdots  \left(\int_{\bR} |f(x_1, x_2,\cdots x_n)|^{p_1} dx_1 \right)^{\frac{p_2}{p_1}} dx_2 \cdots    \right)^{\frac{p_{n-1}}{p_{n-2}}}dx_{n-1}\right )^{\frac{p_{n}}{p_{n-1}}} dx_n \right)^{\frac{1}{p_n}} <\infty.
\end{split}
\]
Similar definitions can be also formulated if some of the indices in $\{p_1, p_2, \cdots, p_n\}$ are equal to $\infty$.  Note that it follows directly from the definition that if $p=p_1 = p_2=\cdots = p_n$, then $L_{p_1p_2\cdots p_n}(\bR^n)$ is the same as the usual Lebesgue space $L_p(\bR^n)$. 

To clearly explain our ideas as well as to understand the importance of the mixed norm Lebesgue spaces, let us consider the following example about a function that is decaying to zero at different rates as $|x| \rightarrow \infty$. Similar examples can be easily produced about different rates of singularity of functions at some certain points.  We consider a bounded measurable function $f: \bR^3 \rightarrow$ satisfying
\begin{equation} \label{example}
|f(x)| \leq \frac{N}{|x_1|^{10^{-k}}|x'|^{10^{k}}}, \quad x =(x_1, x') \in \bR \times \bR^2, \quad |x| >1,
\end{equation}
with some given constant $N>0$ and $k \in \mathbb{N}$ which could be very large.  It can be seen that $f \in L_{p_1p_2p_3}(\bR^3)$ with $p_1 > 10^{k}$ and $p_2 =p_3 > \frac{2}{10^k}$. However, if we consider the usual Lebesgue space, then $f \in L_{p}(\bR^3)$ only if $p > 10^k$, which can be very large when we choose $k$ sufficiently large. In other words, the very fast decaying directions in $(x_2,x_3)$-variable of the function $f$ is completely invisible in the usual unmixed Lebesgue spaces. As a consequence, in the unmixed spaces, the class of functions $f$ as in \eqref{example} is viewed the same as the class of extremely slow decaying functions with
\begin{equation*} 
|f(x)| \leq \frac{N}{|x|^{10^{-k}}}, \quad |x| >1.
\end{equation*}



Now, it is surprisingly interesting to note that for  given numbers $p_1, p_2,\cdots, p_n \in [1, \infty]$, the mixed-norm space $L_{p_1p_2\cdots p_n}(\bR^n)$ is invariant under the scaling \eqref{scaling} if and only if
\begin{equation} \label{p-cond}
\frac{1}{p_1} + \frac{1}{p_2} + \cdots + \frac{1}{p_n} =1.
\end{equation}
The class of critical mixed-norm spaces $L_{p_1p_2\cdots p_n}(\bR^n)$ such that \eqref{p-cond} holds is  the one we will establish the well-posedness for solutions of the Navier-Stokes equations \eqref{NS.eqn} in this paper.  Note that in the special case when $p_1 = p_2=\cdots=p_n$ and \eqref{p-cond} holds, we have $p_1 = p_2=\cdots = p_n =n$. On the other hand,  we also note, as an example, that for the class of functions as in \eqref{example}, it is possible to choose $p_1>10^k$ and $p_2 = p_3 > 2$ but sufficiently close to $2$ so that the triple $(p_1, p_2, p_3)$ satisfies the condition \eqref{p-cond}.  Therefore, in some  certain sense, this paper can be considered as a natural but completely non-trivial extension of the work \cite{Kato}.

Before stating our results, let us introduce some notations used in the paper. For given $p_k \in [1, \infty)$, we write $PL_{p_1p_2\cdots p_n}(\bR^n)$ the space of all vector fields $f \in L_{p_1p_2\cdots p_n}(\bR^n)$ such that
\[
\text{div}(f) =0  \quad \text{in} \quad \bR^n \quad \text{in the sense of distributions}.
\]
Also, for  given $p = (p_1, p_2,\cdots, p_n)$ and $q = (q_1, q_2,\cdots, q_n)$ such that $p_k \in (1, \infty)$ and $q_k \in [p_k, \infty)$ for all $k =1, 2, \dots, n$. Assume that \eqref{p-cond} holds and 
\begin{equation} \label{q-cond}
\frac{1}{q_1} + \frac{1}{q_2} + \cdots + \frac{1}{q_n} = \delta \in (0, 1).
\end{equation}
Then, with given $T \in (0, \infty]$, we denote $\mathcal{X}_{p, q, T}$ the space consisting of all measurable vector field functions $f: \bR^n \times [0, T) \rightarrow \bR^n$ such that 
for
\[
g(x,t) := t^{(1-\delta)/2} f(x,t), \quad \tilde{g}(x,t) := t^{\frac{1}{2}}D_x f (x,t), \quad \text{for} \quad (x,t) \in \bR^n \times (0, T)
\]
then
\[
g  \in \text{C}([0, T), PL_{q_1q_2\cdots q_n}(\bR^n)), \quad \tilde{g} \in  \text{C}([0, T), PL_{p_1p_2\cdots p_n}(\bR^n)) 
\]
and moreover $g(x,0) = 0, \tilde{g}(x,0) =0$ and the norm
\begin{equation} \label{X-spc.def}
\norm{f}_{{\cX}_{p,q, T}} = \sup_{t \in (0, T)} \Big[ \norm{g(\cdot, t)}_{L_{q_1q_2\cdots q_n}(\bR^n)}  +  \norm{\tilde{g}}_{L_{p_1p_2\cdots p_n}(\bR^n)} \Big]<\infty. 
\end{equation}
We also denote $\cY_{p, T}$ the space consisting of all vector field functions $f \in C([0, T), PL_{p_1p_2\cdots p_n}(\bR^n))$ such that $t^{1/2} D_x f \in C([0, T), PL_{p_1p_2\cdots p_n}(\bR^n))$ and
\begin{equation} \label{Y-spc.def}
\norm{f}_{\mathcal{Y}_{p, T}} = \sup_{t \in (0, T)} \Big[ \norm{f(t)}_{L_{p_1p_2\cdots p_n}(\bR^n)} + t^{1/2} \norm{D_xf(t)}_{L_{p_1p_2\cdots p_n}(\bR^n)} \Big] <\infty.
\end{equation}


%

The following theorem on local and global well-posedness of the Cauchy problem \eqref{NS.eqn} in the critical mixed-norm Lebesgue spaces  $L_{p_1p_2\cdots p_n}(\bR^n)$ is the main result of the paper.
\begin{theorem} \label{main-thrm}  Let $p= (p_1, p_2,\cdots, p_n)$ and $q = (q_1, q_2, \cdots, q_n)$. Assume that $p_k \in (2, \infty)$ and $q_k \in [p_k, \infty)$ for all $k = 1, 2,\cdots, n$. Assume also that  \eqref{p-cond} and \eqref{q-cond} hold. Then, there exist a sufficiently small constant $\lambda_0>0$ and a large number $N_0>0$ depending only on $n$ and $p, q$  such that the following assertions hold.
\begin{itemize}
\item[\textup(i)] For every $a_0 \in L_{p_1p_2\cdots p_n}(\bR^n)^n$ with $\nabla \cdot a_0 =0$,  if $\norm{a_0}_{L_{p_1p_2\cdots p_n}(\bR^n)} \leq \lambda_0$, then the Cauchy problem \eqref{NS.eqn} has unique global time solution $u \in \cX_{p,q,\infty} \cap \cY_{p, q, \infty}$ with
\[
\begin{split}
& \norm{u}_{\cX_{p,q,\infty}} \leq N_0 \norm{a_0}_{ L_{p_1p_2\cdots p_n}(\bR^n)} \quad \text{and} 
\\
&\norm{u}_{\cY_{p, q, \infty}} \leq N_0\norm{a_0}_{ L_{p_1p_2\cdots p_n}(\bR^n)}.
\end{split}
\]
\item[\textup{(ii)}] For every $a_0 \in L_{p_1p_2\cdots p_n}(\bR^n)^n$ with $\nabla \cdot a_0 =0$, there exists $T_0>0$ sufficiently small depending on $n, p, q$ and $a_0$ such that the Cauchy problem \eqref{NS.eqn} has unique local time solution $u \in \cX_{p,q, T_0} \cap \cY_{p, q, T_0}$ with 
\[
\begin{split}
& \norm{u}_{\cX_{p,q, T_0}} \leq N_0 \norm{a_0}_{ L_{p_1p_2\cdots p_n}(\bR^n)} \quad \text{and} \\
& \norm{u}_{\cY_{p, q, T_0}} \leq N_0\Big[ \norm{a_0}_{ L_{p_1p_2\cdots p_n}(\bR^n)}+ \norm{a_0}^2_{ L_{p_1p_2\cdots p_n}(\bR^n)} \Big].
\end{split}
\]
\end{itemize}
\end{theorem}

To the best of our knowledge, this is the first time that the kinds of solutions of Navier-Stokes equations in critical mixed-norm Lebesgue spaces are discovered.  As demonstrated in the example in \eqref{example} and the discussion after \eqref{p-cond}, it is possible to choose some of $p_1, p_2,\cdots, p_n$ to be very large numbers so that the given numbers $(p_1, p_2, \cdots, p_n)$ still satisfy the condition \eqref{p-cond}. Due to this reason, in some directions, the class of the initial data and the solutions in Theorem \ref{main-thrm} could decay significantly slow. Similarly, some of the singularity rates in some spatial directions could be very strong. Hence,  our class of solutions  may not belong to $L_n(\bR^n)$ nor $L_2(\bR^n)$, and the solutions obtained in Theorem \ref{main-thrm} may not belong to the classes of solutions found in the papers \cite{Kato, Leray, Leray-1}. Observe also that if $p_1 = p_2=\cdots=p_n$ and \eqref{p-cond} holds, then $p_1 = p_2=\cdots = p_n =n$. In this sense, this paper can be considered as a natural, but completely non-trivial extension of the work \cite{Kato}.

Now, we summarize the above discussion with the following remarks regarding  Theorem \ref{main-thrm}.
\begin{remark} \label{main-remark} The following interesting points are worth highlighting.
\begin{itemize}
\item[\textup{(i)}] Under the condition \eqref{p-cond}, the initial data and the solutions obtained in Theorem \ref{main-thrm} may decay to zero very slow as $|x| \rightarrow \infty$. Similarly, they could also be strongly singular in some spatial directions. Therefore, the solutions obtained in Theorem \ref{main-thrm} may not be in $L_{n}(\bR^n)$ nor $L_2(\bR^n)$.  Consequently, these solutions may not be the same as the ones obtained  in \cite{Kato, Leray, Leray-1}.
\item[\textup{(ii)}]   Let $p_0 \in ( \max\{p_1, p_2, \cdots, p_n, n\}, \infty)$, where $p_1, p_2, \cdots, p_n$ are as in Theorem \ref{main-thrm}. Then, if $a_0 \in L_{p_1p_2\cdots p_n}(\bR^n)$, it follows from the characterization of Besov spaces with negative regularity (see Remark \ref{Besov-space} below) that $a_0 \in \dot B^{-1+\frac{n}{p_0}}_{p_0,\infty}(\bR^n) \subset \textup{BMO}^{-1}(\bR^n)$. In view of this and the results obtained  in \cite{Cannone-2, Cannone-Planchon, Koch-Tataru, F-Planchon} and \cite[Theorem 5.40, p. 234]{Chemin}, Theorem \ref{main-thrm} can be seen as a refinement of these results regarding the persistence of the anisotropic properties of the initial data under the evolution of the Navier-Stokes equations. See also \cite[Section 3.3]{Brandolese} for some different but related results.
\end{itemize}
\end{remark}

To prove Theorem \ref{main-thrm}, we follow the approach developed in \cite{Fu-Ka-1, Fu-Ka, Kato} and in \cite{Giga-1, Giga,  Koch-Tataru, Meyer}. To implement the method, several important and fundamental analysis estimates in mixed norm Lebesgue spaces are developed.  In Section \ref{pre-sec}, we develop and prove a version of Young's inequality in mixed norm Lebesgue spaces. We then use Young's inequality in  mixed norm Lebesgue spaces to establish the time decaying estimates for solutions of heat equations in  mixed norm Lebesgue spaces. The boundedness of the Riesz transform and the boundedness of the Helmholtz-Leray projection in mixed-norm Lebesgue spaces are also established and proved in Section \ref{pre-sec}.  Clearly, these analysis inequalities and estimates are independently topics of great interests and they can be useful in many other problems. To the best of our knowledge,  this paper is the first time that those mixed-norm analysis estimates are developed.  Therefore, besides the great interest  and contribution of our study in the Navier-Stokes equations, the contribution in real and harmonic analysis theory of this paper is also very significant.  The paper concludes with Section \ref{proof-main-result} which provides the proof of Theorem \ref{main-thrm}.

\section{Preliminaries on Analysis inequalities in mixed-norm Lebesgue spaces} \label{pre-sec}

This section gives some main ingredients for the proof of the main theorems in the paper. In particular, we develop Young's inequality in mixed-norm Lebesgue spaces, time decaying rate estimates for solutions of the Cauchy problem for the heat equation in mixed-norm Lebesgue spaces, and Helmholtz-Leray projection in mixed-norm Lebesgue spaces. These results are not only new, fundamental, but they are topics of independent interests and could be useful for many other purposes. For your convenience, we recall that for $p_1, p_2, \cdots p_n \in [1, \infty)$, and for a measurable function $f: \bR^n \rightarrow \bR$, we say that $f$ is in the mixed-norm Lebesgue space $L_{p_1p_1\cdots p_n}(\bR^n)$ if  its norm
\[
\begin{split}
& \norm{f}_{L_{p_1p_1\cdots p_n}(\bR^n)} \\
& = \left( \left( \left(\cdots  \left(\int_{\bR} |f(x_1, x_2,\cdots x_n)|^{p_1} dx_1 \right)^{\frac{p_2}{p_1}} dx_2 \cdots    \right)^{\frac{p_{n-1}}{p_{n-2}}}dx_{n-1}\right )^{\frac{p_{n}}{p_{n-1}}} dx_n \right)^{\frac{1}{p_n}} < \infty.
\end{split}
\]
Similar definitions can be also formulated if some of the indices $\{p_1, p_2, \cdots, p_n\}$ are equal to $\infty$. As we already discussed, the significant role of the the mixed-norm Lebesgue space $L_{p_1p_1\cdots p_n}(\bR^n)$ is that it captures very well the functions that are singular at certain points or decaying to zero as $|x| \rightarrow \infty$ with different rates in different $x$-directions.
\subsection{Young's inequality in mixed norm Lebesgue spaces} This subsection establishes the following new result on Young's inequality in mixed norm Lebesgue spaces. The result will be useful in the study of heat equations in mixed norm Lebesgue spaces. Our theorem can be stated as in the following.
\begin{theorem}[Young's inequality in mixed norm] \label{Young} Let $p_k, r_k$ and $q_k$ be given numbers in $[1,\infty]$ that satisfy
\[
\frac{1}{p_k} + 1 = \frac{1}{q_k} + \frac{1}{r_k}, \quad k = 1, 2,\cdots, n.
\]
Then
\begin{equation} \label{young-inq}
\norm{f*g}_{L_{p_1p_2\cdots p_n}(\bR^n)} \leq \norm{f}_{L_{q_1q_2\cdots q_n}(\bR^n)} \norm{g}_{L_{r_1r_2\cdots r_n}(\bR^n)}
\end{equation}
for every $f \in L_{q_1q_2\cdots q_n}(\bR^n)$ and $g \in L_{r_1r_2\cdots r_n}(\bR^n)$.
\end{theorem} 
\begin{proof}  We use induction on $n$. Observe that when $n =1$, the inequality \eqref{young-inq} is the classical Young's inequality. We now assume that the inequality holds true in $(n-1)$-dimension and prove it for $n$-dimension with $n \geq 2$.  Let us denote $p_k', q_k', r_k'$ the H\"{o}lder's conjugates of $p_k, q_k, r_k$ respectively. By the assumption, we see that
\begin{equation} \label{prq-k}
\frac{1}{r_k'} + \frac{1}{p_k} + \frac{1}{q_k'} =1, \quad \frac{r_k}{p_k} + \frac{r_k}{q_k'} =1, \quad \text{and} \quad \frac{q_k}{r_k'} + \frac{q_k}{p_k} = 1, \quad k =1, 2,\cdots, n.
\end{equation}
We split the proof into three different cases.\\ \ \\
{\bf Case I}. We assume that $p_1 <\infty$ and $q_1 < \infty$. In this case, we also see that $r_1 < \infty$. For $x, y \in \bR^n$, we write $x = (x_1, x') \in \bR \times \bR^{n-1}$ and $y = (y_1, y') \in \bR \times \bR^{n-1}$.  As $q_1 <\infty$, by using the last two identities in \eqref{prq-k}, we have
\[
\begin{split}
& |(f*g)(x)|  \leq \int_{\bR^n} |f(y)||g(x-y)| dy  \\
& = \int_{\bR^{n-1}} \left[ \int_{\bR} |f(y_1, y')| |g(x_1 - y_1, x' - y')| dy_1\right] dy' \\
& =\int_{\bR^{n-1}} \left[ \int_{\bR} |f(y_1, y')|^{\frac{q_1}{r_1'}}\Big ( |f(y_1, y')|^{\frac{q_1}{p_1}} |g(x_1 - y_1, x' - y')|^{\frac{r_1}{p_1}} \Big)  |g(x_1-y_1, x' - y')|^{\frac{r_1}{q_1'}} dy_1\right] dy',
\end{split}
\]
Note that in the above inequality also holds when $r_1'= \infty$ with $\frac{1}{r_1'} =0$.  Now, by using the first identity in \eqref{prq-k} and the H\"{o}lder's inequality with respect to the integration in $y_1$-variable, we obtain
\begin{equation} \label{1-step-Y.eqn}
\begin{split}
& |(f* g)(x)| \\
& \leq \int_{\bR^{n-1}} \left[ |f_1(y')|^{\frac{q_1}{r_1'}} |g_1(x'-y')|^{\frac{r_1}{q_1'}} \left(\int_{\bR}|f(y_1,y')|^{q_1} |g(x_1-y_1, x'-y')|^{r_1} dy_1 \right)^{\frac{1}{p_1}}\right] dy',
\end{split}
\end{equation}
where we denote
\begin{equation} \label{f-g-1}
\begin{split}
& f_1(y') = \left(\int_{\bR}|f(y_1, y')|^{q_1} dy_1 \right)^{\frac{1}{q_1}}\quad \text{and} \\
& g_1(y') = \left(\int_{\bR} |g(y_1, y')|^{r_1} dy_1 \right)^{\frac{1}{r_1}} , \quad \text{for a.e.} \quad y' \in \bR^{n-1}.
\end{split}
\end{equation}
As $p_1 <\infty$, it follows from \eqref{1-step-Y.eqn} that
\[
\begin{split}
& G(x') := \left( \int_{\bR} |(f*g)(x_1, x')|^{p_1} dx_1 \right)^{\frac{1}{p_1}}\\
& \leq \left\{\left(\int_{\bR^{n-1}} \left[ |f_1(y')|^{\frac{q_1}{r_1'}} |g_1(x'-y')|^{\frac{r_1}{q_1'}} \left(\int_{\bR}|f(y_1,y')|^{q_1} |g(x_1-y_1, x'-y')|^{r_1} dy_1 \right)^{\frac{1}{p_1}}\right] dy' \right)^{p_1} dx_1 \right\}^{\frac{1}{p_1}}.
\end{split}
\]
From this, and by using the Minskowski's inequality, we see that
\[
\begin{split}
& \left( \int_{\bR} |(f*g)(x_1, x')|^{p_1} dx_1 \right)^{\frac{1}{p_1}}\\
& \leq \int_{\bR^{n-1}} \left ( |f_1(y')|^{\frac{q_1}{r_1'}} |g_1(x'-y')|^{\frac{r_1}{q_1'}} I(x', y') \right) dy' ,
\end{split} 
\]
where
\[
I (x', y')= \left( \int_{\bR}\left(\int_{\bR}|f(y_1,y')|^{q_1} |g(x_1-y_1, x'-y')|^{r_1} dy_1 \right)  dx_1 \right)^{\frac{1}{p_1}}.
\]
By the Fubini's theorem, we see that
\[
\begin{split}
I (x', y') & = \left( \int_{\bR}\left(\int_{\bR}|f(y_1,y')|^{q_1} |g(x_1-y_1, x'-y')|^{r_1} dx_1 \right)  dy_1 \right)^{\frac{1}{p_1}} \\
& = \norm{g(\cdot, x'-y')}_{L^{r_1}(\bR)}^{\frac{r_1}{p_1}} \norm{f(\cdot, y')}_{L_{q_1}(\bR)}^{\frac{q_1}{p_1}} \\
&= |g_1(x'-y')|^{\frac{r_1}{p_1}} |f_1(y')|^{\frac{q_1}{p_1}}.
\end{split} 
\]
Therefore,
\[
G(x')  \leq \int_{\bR^{n-1}} \left ( |f_1(y')|^{\frac{q_1}{r_1'} + \frac{q_1}{p_1}} |g_1(x'-y')|^{\frac{r_1}{q_1'} + \frac{r_1}{p_1}}  \right) dy'.
\]
From this, and by using the last two identities in \eqref{prq-k}, we obtain
\[
G(x') = (f_1 * g_1) (x').
\]
Then, by induction hypothesis, we see that
\[
\begin{split} 
\norm{f*g}_{L_{p_1 p_2\cdots p_n}(\bR^n)} & = \norm{G}_{L_{p_2p_3\cdots p_n}(\bR^{n-1})} \leq \norm{g_1}_{L_{r_2r_3\cdots r_n}(\bR^{n-1})} \norm{f_1}_{L_{q_2q_3\cdots q_n}(\bR^{n-1})}\\
&= \norm{g}_{L_{r_1r_2\cdots r_n}(\bR^{n-1})} \norm{f}_{L_{q_1q_2\cdots q_n}(\bR^{n-1})}.
\end{split}
\]
This proves the desired estimate for the case $q_1<\infty$ and $p_1< \infty$. \\ \ \\
{\bf Case II}.  We assume that $p_1 =\infty$ and $q_1 <\infty$. In this case, we observe that $r_1' = q_1 \in [1, \infty)$. In this case, we write
\[
|f*g(x)| \leq  \int_{\bR^{n-1}} \left[ \int_{\bR} |f(y_1, y')| |g(x_1 - y_1, x' - y')| dy_1\right] dy'.
\]
If $r_1<\infty$, as $\frac{1}{q_1} + \frac{1}{r_1} =1$, we apply the H\"{o}lder's inequality for the integration with respect to $y_1$ to obtain
\[
|(f*g)(x)| \leq \int_{\bR^{n-1}} f_1(y') g(x'-y') dy' = (f_1* g_1)(x'), \quad \text{for a.e.} \ x= (x_1, x') \in \bR \times \bR^{n-1},
\]
where $f_1, g_1$ are defined as in \eqref{f-g-1}.  Observe also that the similar estimate can be also done when $r_1 =\infty$. From this, the desired inequality follows by the induction hypothesis as in {\bf Case I}. The proof is of this case therefore completed.\\ \ \\
{\bf Case III}. We are left to consider the case that $q_1 = \infty$. In this case, it follows that $p_1 = \infty$ and $r_1 =1$. By defining
\[
G(x') = \sup_{x_1 \in \bR}| (f* g)(x_1, x')|, \quad f_1 (x') = \sup_{x_1 \in \bR}|f(x_1, x')|, \quad g_1(x') = \int_{\bR} |g(x_1, x')| dx_1
\]
we see that
\[
G(x') \leq (f_1* g_1) (x').
\]
Then, we also obtain the same desired estimate. The proof is then completed. \\
\end{proof}
\begin{remark} Theorem \ref{Young} gives the classical unmixed-norm Young's inequality when $p_1=p_2=\cdots = p_n$ and $q_1 = q_2=\cdots = q_n$.
\end{remark} 
\subsection{Heat equations in mixed norm Lebesgue spaces} This subsection develops estimates of  time decaying rates for solutions of heat equations in mixed norm Lebesgue spaces. We consider the Cauchy problem for the heat equation
\begin{equation} \label{heat.eqn}
\left\{
\begin{array}{cccl}
u_t -\Delta u & =& 0 & \quad \text{in} \quad \bR^n \times (0, \infty),\\
u_{|t =0} & = & u_0 & \quad \text{on} \quad \bR^n.
\end{array}
\right.
\end{equation}
Under some suitable conditions on the initial data $u_0$, it is well known that 
\begin{equation} \label{sol-heat.eqn}
u(x,t) = e^{\Delta t} u_0 (x) = (G_t * u_0) (x), \quad (x,t) \in \bR^n \times (0, \infty),
\end{equation}
is a solution of \eqref{heat.eqn},  where
\[
G_t(x) = \frac{1}{(4\pi t)^{\frac{n}{2}}} e^{-\frac{|x|^2}{4t}}, \quad \quad (x,t) \in \bR^n \times (0, \infty).
\]
The following new and fundamental result on the time decaying rates of the solutions \eqref{sol-heat.eqn} of the heat equation \eqref{heat.eqn} in mixed norm Lebesgue spaces is the main result of this subsection. 
\begin{theorem}[Time decaying of solutions for heat equation in mixed-norm] \label{heat-decay} Let $1 \leq q_k \leq p_k \leq \infty$. There exists a positive constant $N$ depending only on $p_1, p_2,\cdots, p_n, q_1, q_2,\cdots, q_n$ such that for every solution $u(x,t)= e^{\Delta t}u_0(x)$ defined in \eqref{sol-heat.eqn} of the Cauchy problem \eqref{heat.eqn}    with $u_0 \in L_{q_1q_2\cdots q_3} (\bR^n)$, then for $ t >0$
\begin{equation} \label{sol-decay}
\norm{u(\cdot, t)}_{L_{p_1p_2\cdots p_n}(\bR^n)} \leq  N t^{ -\frac{1}{2}  \sum_{k=1}^n  (\frac{1}{q_k} -\frac{1}{p_k} )} \norm{u_0}_{L_{q_1q_2\cdots q_n}(\bR^n)}.
\end{equation}
Moreover, for every $l = 1, 2,\cdots$ and  for $ t >0$
\begin{equation} \label{D-sol-decay}
\norm{D_{x}^l u(\cdot, t)}_{L_{p_1p_2\cdots p_n}(\bR^n)} \leq  N t^{-\frac{l}{2}-\frac{1}{2}  \sum_{k=1}^n  (\frac{1}{q_k} -\frac{1}{p_k} )} \norm{u_0}_{L_{q_1q_2\cdots q_n}(\bR^n)},
\end{equation}
where $D^{l}_{x}$ denotes the $l^{\text{th}}$-derivative in $x$-variable.
\end{theorem}
\begin{proof}  We begin with the proof of \eqref{sol-decay}. For each $k = 1, 2,\cdots, n$, by the assumption that $q_k \leq p_k$, we can find $r_k \in [1, \infty]$ such that
\begin{equation} \label{rk.def}
\frac{1}{p_k} + 1 = \frac{1}{r_k} + \frac{1}{q_k}.
\end{equation}
Then, because $u(x,t) = (G_t * u_0)(x)$, we can use the mixed-norm Young's inequality in Theorem \ref{Young} to see that
\begin{equation} \label{Y-heat}
\norm{u(\cdot, t)}_{L_{p_1p_1\cdots p_n}(\bR^n)} \leq \norm{G_t(\cdot)}_{L_{r_1r_2\cdots r_n}(\bR^n)} \norm{u_0}_{L_{q_1q_2\cdots q_n}(\bR^n)}.
\end{equation}
We now note that we can write $G_t$ as
\[
G_t(x) = g_t(x_1) g_t(x_2) \cdots g_t (x_n) \quad \text{for} \quad x = (x_1,x_2,\cdots, x_n) \in \bR^n, \quad t >0
\]
where $g_t$ is the heat kernel in $\bR$:
\begin{equation} \label{g-t.def}
g_t(s) = \frac{1}{\sqrt{4\pi t}} e^{-\frac{s^2}{4t}}, \quad s \in \bR, \quad t>0.
\end{equation}
Note also that for each $r \in [1, \infty)$ we have
\[
\begin{split}
\norm{g_t(\cdot)}_{L_r(\bR)} & = \frac{1}{\sqrt{4\pi t}}  \left(\int_{\bR} e^{-\frac{rs^2}{4t}} \right)^{\frac{1}{r}} = \frac{1}{\sqrt{4\pi t}} \left(\sqrt{\frac{4t}{r}}\right)^{\frac{1}{r}} \left( \int_{\bR} e^{-z^2} dz \right)^{\frac{1}{r}} \\
& = N(r) t^{-\frac{1}{2}(1- \frac{1}{r})}.
\end{split}
\]
On the other hand, we also see that
\[
\norm{g_t(\cdot)}_{L_\infty(\bR)} \leq N t^{-\frac{1}{2}}, \quad t >0.
\]
Therefore, we conclude that for every $r \in [1, \infty]$
\begin{equation} \label{g-t-norm}
\norm{g_t(\cdot)}_{L_\infty(\bR)} \leq N(r) t^{-\frac{1}{2}(1- \frac{1}{r})}, \quad t >0.
\end{equation}
From this, we infer that
\[
\begin{split}
\norm{G_t(\cdot)}_{L_{r_1r_2\cdots r_n} (\bR^n)} & = \norm{g_t(\cdot)}_{L_{r_1}(\bR)} \norm{g_t(\cdot)}_{L_{r_2}(\bR)} \cdots \norm{g_t(\cdot)}_{L_{r_n}(\bR)} \\
& =N(r_1, r_2,\cdots, r_n) t^{-\frac{1}{2}(n - \sum_{i=1}^n \frac{1}{r_i})}\\
& = N(p_1, p_2,\cdots p_n, q_1, q_2,\cdots q_n) t^{-\frac{1}{2}\sum_{i=1}^n (\frac{1}{q_i} - \frac{1}{p_i})}, \quad t >0,
\end{split}
\]
where we have used \eqref{rk.def} in the last estimate. This last estimate together with \eqref{Y-heat} implies \eqref{sol-decay}.

Next, we prove \eqref{D-sol-decay}. We only demonstrate the proof of \eqref{D-sol-decay} with $l =1$ as  the general case can be done in a similar way. We observe that for each $i = 1, 2,\cdots, n$
\[
D_{x_i}u(x,t) = ([D_{x_i} G_t]* u_0)(x), \quad (x,t) \in \bR^n \times (0, \infty).
\]
Then, by the mixed norm Young's inequality in Theorem \ref{Young}, we have
\begin{equation} \label{Y-heat-2}
\norm{D_{x_i} u(\cdot, t)}_{L_{p_1p_2\cdots p_n}(\bR^n)} \leq \norm{D_{x_i}G_t(\cdot)}_{L_{r_1r_2\cdots r_n}(\bR^n)} \norm{u_0(\cdot)}_{L_{q_1q_2\cdots q_n}(\bR^n)}.
\end{equation}
It remains to estimate the mixed norm $\norm{D_{x_i}G_t(\cdot)}_{L_{r_1r_2\cdot r_n}(\bR^n)}$. Note that
\[
D_{x_i}G_t(x) =- \frac{1}{(4\pi t)^{\frac{n}{2}}} \frac{x_i}{2t} e^{-\frac{|x|^2}{4t}}, \quad x = (x_1,x_2,\cdots, x_n) \in \bR^n\quad t >0. 
\]
Consequently,
\[
D_{x_i}G_t(x) = h_t(x_i)  \left(\displaystyle{ \prod_{k\not= i} g_t(x_k)} \right), \quad x = (x_1,x_2,\cdots, x_n) \in \bR^n\quad t >0,
\]
where  $g_t$ is defined as in \eqref{g-t.def} and
\[
h_t(s) =-\frac{1}{\sqrt{4\pi t}} \frac{s}{2t} e^{-\frac{s^2}{4t}}, \quad s \in \bR \quad \text{and} \quad t>0. 
\]
We observe that 
\[
|h_t(s)| \leq \frac{N}{t} \frac{|s|}{\sqrt{4t}} e^{-\frac{s^2}{4t}}  = \frac{N}{t} |z| e^{-|z|^2}, \quad \text{where} \quad z = \frac{s}{\sqrt{4t}}.
\]
As $|z| e^{-z^2}$ is a bounded function for $z \in \bR$, we conclude that
\[
\norm{h_t}_{L_\infty(\bR)} \leq \frac{N}{t}, \quad t >0.
\]
On the other hand, if $r_i \in [1,\infty)$, we see that
\[
\begin{split}
\norm{h_t}_{L_{r_i}(\bR)} & = \frac{N(r_i)}{t} \left [\int_{\bR} \Big ( \frac{\sqrt{r_i}|s|}{\sqrt{4t}} \Big)^{r_i} e^{\frac{-r_i s^2}{4t}} ds \right]^{\frac{1}{r_i}} \\
& =  N(r_i) t^{-\frac{1}{2} -  \frac{1}{2}(1 - \frac{1}{r_i})} \left[ \int_{\bR} |z|^{r_i} e^{-z^2} dz \right]^{\frac{1}{r_i}} \\
& = N(r_i) t^{-\frac{1}{2} -  \frac{1}{2}(1 - \frac{1}{r_i})} .
\end{split}
\]
Therefore, we conclude that for every $r_i \in [1, \infty]$ 
\[
\norm{h_t}_{L_{r_i}(\bR)} = N (r_i)  t^{-\frac{1}{2} -  \frac{1}{2}(1 - \frac{1}{r_i})}, \quad t >0.
\]
From this estimate and \eqref{g-t-norm}, we see that
\[
\begin{split}
\norm{D_{x_i} G_t}_{L_{r_1r_2\cdots r_n}(\bR^n)} & = \norm{h_t}_{L_{r_i}(\bR)} \prod_{k\not=i} \norm{g_t}_{L_{r_k}(\bR)} \\
& \leq N(r_1, r_2,\cdots, r_n) t^{-\frac{1}{2} -\frac{1}{2} (n- \sum_{k=1}^n \frac{1}{r_k})}.
\end{split}
\]
From this, and by using \eqref{rk.def}, we infer that
\[
\norm{D_{x_i} G_t}_{L_{r_1r_2\cdots r_n}(\bR^n)} \leq N(r_1, r_2,\cdots, r_n) t^{-\frac{1}{2} -\frac{1}{2} \sum_{k=1}^n ( \frac{1}{q_k} - \frac{1}{p_k})}.
\]
This last estimate and \eqref{Y-heat-2} imply \eqref{D-sol-decay} with $l=1$. The proof of the lemma is complete.
\end{proof}
Next, we introduce and prove the following simple lemma on the continuity property of the solutions of the heat equation \eqref{heat.eqn} in mixed norm spaces. The result will be useful in the paper.
\begin{lemma} \label{cont-lemma} For each $k = 1, 2, \cdots, n$, let $p_k \in [1, \infty)$. Assume that $u_0 \in L_{p_1p_2\cdots p_n}(\bR^n)$. Let $u(x,t) = e^{\Delta t} u_0$ be the solution of the heat equation \eqref{heat.eqn} defined in \eqref{sol-heat.eqn}. Then,  $u: C([0, \infty), L_{p_1p_2\cdots p_n}(\bR^n))$ and
\begin{equation} \label{limit-zero}
\lim_{t\rightarrow 0^+} \norm{u(\cdot, t) - u_0}_{L_{p_1p_2\cdots p_n}(\bR^n)} =0.
\end{equation}
\end{lemma}
\begin{proof}  We only need to prove \eqref{limit-zero}, as the proof of the continuity of $u$ at $t_0>0$ can be done similarly.  Let $\epsilon >0$, by using the truncation and a multiplication by a suitable cut-off function, we can find a bounded compactly support function $\tilde{u}_0$ such that
\[
\norm{u_0 - \tilde{u}_0}_{L_{p_1p_2\cdots p_n}(\bR^n)} \leq \frac{\epsilon}{4N_0},
\] 
where $N_0 = N_0(n, p_1, p_2,\cdots, p_n) >1$ is the number defined in Theorem \ref{heat-decay}. Precisely, by  Theorem \ref{heat-decay}, we have
\[
\norm{e^{\Delta t} (u_0 -\tilde{u}_0)}_{L_{p_1p_2\cdots p_n}(\bR^n)} \leq N_0 \norm{u_0 - \tilde{u}_0}_{L_{p_1p_2\cdots p_n}(\bR^n)} \leq \frac{\epsilon}{4}.
\]
From the previous two estimates, we see that
\begin{equation}\label{lim-appr}
\norm{e^{\Delta t} (u_0 -\tilde{u}_0) - (u_0 -\tilde{u}_0)}_{L_{p_1p_2\cdots p_n}(\bR^n)} \leq \frac{\epsilon}{4} + \frac{\epsilon}{4N} \leq \frac{\epsilon}{2}.
\end{equation}
Our next goal is to show that
\[
\lim_{t\rightarrow 0^+} \norm{e^{\Delta t} \tilde{u}_0 - \tilde{u}_0}_{L_{p_1p_2\cdots p_n} (\bR^n)} =0.
\]
Take $p > \max\{p_1, p_2,\cdots, p_n\}$ and choose the numbers  $q_k \in (p_k, \infty)$ such that
\[
 \frac{1}{q_k} = \frac{1}{p_k} - \frac{1}{p}, \quad k = 1, 2,\cdots, n.
\]
Then, by applying the H\"{o}lder's inequality repeatedly for each integration with respect to each variable $x_k$, we see that
\[
\begin{split}
& \norm{e^{\Delta t} \tilde{u}_0 - \tilde{u}_0}_{L_{p_1p_2\cdots p_n} (\bR^n)}\\
& \leq \norm{e^{\Delta t} \tilde{u}_0 - \tilde{u}_0}_{L_{p} (\bR^n)} \norm{e^{\Delta t} \tilde{u}_0 - \tilde{u}_0}_{L_{q_1q_2\cdots q_n} (\bR^n)} \\
& \leq \norm{e^{\Delta t} \tilde{u}_0 - \tilde{u}_0}_{L_{p} (\bR^n)} \Big [\norm{e^{\Delta t} \tilde{u}_0}_{L_{q_1q_2\cdots q_n} (\bR^n)}+ \norm{\tilde{u}_0}_{L_{q_1q_2\cdots q_n} (\bR^n)} \Big] \\
& \leq  N \norm{e^{\Delta t} \tilde{u}_0 - \tilde{u}_0}_{L_{p} (\bR^n)} \norm{\tilde{u}_0}_{L_{q_1q_2\cdots q_n} (\bR^n)}.
\end{split}
\]
Observe that as $\tilde{u}_0$ is bounded and compactly supported, $\norm{\tilde{u}_0}_{L_{q_1q_2\cdots q_n} (\bR^n)} <\infty$. Therefore,
\[
 \norm{e^{\Delta t} \tilde{u}_0 - \tilde{u}_0}_{L_{p_1p_2\cdots p_n} (\bR^n)}\leq \tilde{N}\norm{e^{\Delta t} \tilde{u}_0 - \tilde{u}_0}_{L_{p} (\bR^n)}  \rightarrow 0 \quad \text{as} \quad t\rightarrow 0^+,
\]
where in the last assertion, we used the classical result of the continuity of the heat flow in $L_p(\bR^n)$ and the fact that $\tilde{u}_0 \in L_{p}(\bR^n)$.  From this and \eqref{lim-appr}, we conclude that there is $\delta = \delta(\epsilon) >0$ such that
\[
\norm{e^{\Delta t} u_0  - u_0 }_{L_{p_1p_2\cdots p_n}(\bR^n)}  \leq \epsilon, \quad \forall \ t \in (0, \delta_0).
\]
This proves \eqref{limit-zero} as desired.
\end{proof}
\begin{remark} It is interesting to note that Theorem \ref{heat-decay} shows the persistence of the anisotropic properties of the initial data under the evolution of the heat equation.  This phenomenon seems to be new. Theorem \ref{heat-decay} and Lemma \ref{cont-lemma} recover the classical results when $p_1=p_2=\cdots = p_n$ and $q_1 = q_2=\cdots = q_n$.
\end{remark} 
\begin{remark} \label{Besov-space} For given numbers $p_1, p_2,\cdots, p_n \in (1,\infty)$ that satisfy \eqref{p-cond}, if $u_0 \in L_{p_1p_2\cdots p_n} (\bR^n)$, by Theorem \ref{heat-decay}, we see that
\[
t^{\frac{1}{2} (1-\frac{n}{p_0})} \norm{e^{\Delta t} u_0}_{L_{p_0}(\bR^n)} \leq N  \norm{u_0}_{L_{p_1p_2\cdots p_n}(\bR^n)},
\]
for $p_0 \in (\max\{p_1, p_2,\cdots, p_n, n\}, \infty)$ and for $N = N(n, p_0, p_1, p_2,\cdots, p_n)$. Then, it follows from the characterization of Besov spaces with negative regularity (see \cite[Theorem 2.34, p. 72]{Chemin}, or \cite[eqn (8.6), p. 177]{L-R}, and also \cite{Cannone, Cannone-2}) that $u_0 \in \dot B_{p_0, \infty}^{-1+\frac{n}{p_0}}(\bR^n)$ with its norm is defined as in \eqref{Besov-norm}
\[
\norm{u_0}_{\dot B_{p_0, \infty}^{-1+\frac{n}{p_0}}(\bR^n)} \approx \sup_{t>0}t^{\frac{1}{2} (1-\frac{n}{p_0})} \norm{e^{\Delta t} u_0}_{L_{p_0}(\bR^n)} <\infty.
\]
In particular, it follows from this and \cite[eqn (23)]{Koch-Tataru} that $u_0 \in \textup{BMO}^{-1}(\bR^n)$.
 \end{remark}
\subsection{Helmholtz-Leray projection in mixed-norm Lebesgue spaces} Let $\mathbb{P} = \text{Id} -\nabla \Delta^{-1}\nabla \cdot$ be the Helmholtz-Leray projection onto the divergence-free vector fields. This subsection proves that 
\[
\norm{\bP(f)}_{L_{p_1p_2\cdots p_n}(\bR^n)} \leq N \norm{f}_{L_{p_1p_2\cdots p_n}(\bR^n)},
\]
for every $f \in L_{p_1p_2\cdots p_n}(\bR^n)^n$ and for $p_1, p_2,\cdots, p_n \in (1, \infty)$.  This estimate is an important ingredient in our paper. To achieve it, we need to recall the following definition of Muckenhoupt  $A_q(\bR^n)$-class of weights, which is needed for the proof of Theorem \ref{Riesz} below. For each $q \in (1, \infty)$, a non-negative, locally integrable function $\omega: \mathbb{R}^n \rightarrow \mathbb{R}$ is said to be in the Muckenhoupt  $A_q(\bR^n)$-class of weights if
\[
[\omega]_{A_q} := \displaystyle{\sup_{R>0 , x_0 \in \bR^n} \left(\frac{1}{|B_R(x_0)|} \int_{B_R(x_0)} \omega(x) dx \right) \left(\frac{1}{B_R(x_0)} \int_{B_R(x_0)} \omega(x)^{-\frac{1}{q-1}} dx  \right)^{q-1}} < \infty,
\]
where $B_R(x_0)$ denotes the ball in $\bR^n$ of radius $R$ centered at $x_0 \in \bR^n$.  In the following, for each given $p \in [1, \infty)$ and each given weight $\omega: \bR^n \rightarrow \bR$, a measurable function $f: \bR^n \rightarrow \bR$ is said to be in the weighted Lebesgue space $L_p(\bR^n, \omega)$ if its norm
\[
\norm{f}_{L_{p}(\bR^n, \omega)} = \left( \int_{\bR^n} |f(x)|^p \omega(x) dx \right)^{\frac{1}{p}} < \infty.
\]
We also recall the following amazing result from \cite[Theorem 6.2]{Krylov-survey}, which is a beautiful application of  the Rubio De Francia extrapolation theory (see \cite{David} for instance).
\begin{theorem} \label{extrapolation-thrm} Let $p_k \in (1,\infty)$ for all $k = 1, 2,\cdots, n$.  Then, there exists a constant $K_0 = K_0(n, p_1, p_2, \cdots, p_n) \geq 1$ such that the following holds true. For a pair of given measurable functions $f, g: \bR^n \rightarrow \bR$ such that if
\[
\norm{f}_{L_{p_1}(\bR^n, \omega)} \leq \norm{g}_{L_{p_1}(\bR^n, \omega)}
\]
for every $\omega \in A_{p_1}$ with $[\omega]_{A_{p_1}} \leq K_0$, then we have
\[
\norm{f}_{L_{p_1p_2\cdots p_n}(\bR^n)} \leq 4^n\norm{g}_{L_{p_1p_2\dots p_n}(\bR^n)}.
\]
\end{theorem}

Now, we begin with the following important result on the boundedness of the Riesz  transform in mixed-norm Lebesgue spaces. Interested readers may find \cite[Corollary 2.7]{Dong-Kim} and \cite[Lemma 2.1]{TP} for other interesting related results in mixed-norm spaces.
\begin{theorem} \label{Riesz} For any $j = 1, 2,\cdots, n$ and any $p_1, p_2,\cdots, p_n \in (1, \infty)$, there exists a positive constant $N = N(p_1, p_2, \cdots, p_n, n)$ such that
\[
\norm{\mathscr{R}_j(f)}_{L_{p_1p_2\cdots p_n}(\bR^n)} \leq N \norm{f}_{L_{p_1p_2\cdots p_n}(\bR^n)}
\]
for every $f \in {L_{p_1p_2\cdots p_n}(\bR^n)}$, where $\mathscr{R}_j$ is the $j^{\textup{th}}$-Riesz transform defined by $\mathscr{R}_j (f) = \partial_{x_j} (-\Delta)^{-\frac{1}{2}} f$.
\end{theorem}
\begin{proof} We plan to apply Theorem \ref{extrapolation-thrm}.  For given $p_1, p_2, \cdots, p_n \in (1, \infty)$, let $K_0$ be as in Theorem \ref{extrapolation-thrm}. By using the truncation and a multiplication with suitable cut-off functions, we can approximate $f \in L_{p_1p_2\cdots p_n} (\bR^n)$ by a sequence of bounded compactly supported functions. Therefore,  we may assume that $f$ is bounded and compactly supported in $\bR^n$. Without loss of generality, we can also assume that $p_1 =\min\{p_1, p_2,\cdots, p_n\}$.  Under these assumptions, we see that $f \in L_{p_1}(\bR^n,\omega)$  for every weight  $\omega \in A_{p_1}$.  Then,  since $p_1 \in (1, \infty)$, by the classical Calder\'{o}n-Zygmund theory (see \cite{David, RF} for instance),  there exists a constant $N = N(p_1, n, K_0)$ such that 
\begin{equation}  \label{Lp-1.weight}
\norm{\mathscr{R}_j(f)}_{L_{p_1}(\bR^n, \omega)} \leq N \norm{f}_{L_{p_1}(\bR^n, \omega)},
\end{equation}
for every $\omega \in A_{p_1}$ with $[\omega]_{A_{p_1}} \leq K_0$.  From \eqref{Lp-1.weight} and  Theorem \ref{extrapolation-thrm},  we infer that
 \[
 \norm{\mathscr{R}_j(f)}_{L_{p_1p_2\cdots p_n}(\bR^n)} \leq 4^n N  \norm{f}_{L_{p_1 p_2\cdots p_n}(\bR^n)}.
\]
This is the desired estimate and the proof is therefore completed.
\end{proof}
The following consequence of Theorem \ref{Riesz} gives the boundedness of the Helmholtz-Leray projection in mixed norm Lebesgue spaces, which is an important ingredient in the paper.
\begin{corollary} \label{H-L-projection} Let $\mathbb{P} = \textup{Id} -\nabla \Delta^{-1}\nabla \cdot$ be the Helmholtz-Leray projection onto the divergence-free vector fields. Let $p_1, p_2,\cdots, p_n \in (1, \infty)$. Then, one has
\[
\norm{\bP(f)}_{L_{p_1p_2\cdots p_n}(\bR^n)} \leq N \norm{f}_{L_{p_1p_2\cdots p_n}(\bR^n)},
\]
for every $f \in L_{p_1p_2\cdots p_n}(\bR^n)^n$, where $N = N(p_1, p_2,\cdots, p_n, n)$ is a positive constant.
\end{corollary}
\begin{proof} Note that with $f = (f_1, f_2,\cdots f_n) \in L_{p_1p_2\cdots p_n}(\bR^n)^n$, we have  $\bP(f) = (\bP(f)_1, \bP(f)_2,\cdots,  \bP(f)_n)$ with
\[
\begin{split}
\bP(f)_k  
& = f_k + \mathscr{R}_{k}  \sum_{j=1}^n \mathscr{R}_j f_j, \quad k = 1, 2\cdots, n,
\end{split}
\]
where $\mathscr{R}_j$ is the $j^{\textup{th}}$-Riesz transform. Therefore, it follows from Theorem \ref{Riesz} that
\[
\norm{\bP(f)}_{L_{p_1p_2\cdots p_n}(\bR^n)} \leq N \norm{f}_{L_{p_1p_2\cdots p_n}(\bR^n)}
\]
which is our desired estimate.
\end{proof}
\section{Navier-Stokes equations in critical mixed-norm Lebesgue spaces} \label{proof-main-result}

This section provides the proof  of Theorem \ref{main-thrm}. We follow the approach introduced in \cite{Fu-Ka-1, Fu-Ka, Kato} and in \cite{Koch-Tataru, Meyer}.  Recall that $\bP$ denotes the Helmholtz-Leray projection which is defined in Corollary \ref{H-L-projection}. By applying $\bP$ on the system \eqref{NS.eqn}, we see that the system \eqref{NS.eqn} is recasted in the abstract way as the following
\begin{equation} \label{abstract-NS.eqn}
\left\{
\begin{array}{cccl}
u_t + \mathcal{A} u + F(u,u) & = & 0 & \quad \text{in} \quad \bR^n \times (0, \infty),\\
u(\cdot, 0) & = & a_0(\cdot)& \quad \text{on} \quad \bR^n,
\end{array}
\right.
\end{equation}
where $\mathcal{A} = -\bP \Delta = -\Delta \bP$ and
\begin{equation} \label{F.def}
F(u,v) = \bP((u\cdot \nabla) v).
\end{equation}
By the Duhamel's principle, the system \eqref{abstract-NS.eqn} is then converted to the following integral equation
\begin{equation} \label{u-abstract.eqn}
u = u_0 + \G (u,u),
\end{equation}
where
\begin{equation} \label{G.def}
u_0(t) = e^{-\mathcal{A} t} a_0, \quad \text{and} \quad \G (u, v)(t) = -\int_{0}^t e^{-(t-s)\A} F(u(s), v(s)) ds.
\end{equation}
To proceed, we need several estimates. We begin with the following lemma on the time decaying properties  for the semi-group $e^{-\A t}$ in mixed norm Lebesgue spaces.
\begin{lemma} \label{linear-est} For  each $k = 1, 2,\cdots, n$, let $1 < p_k \leq q_k < \infty$ be given numbers. Also,  let $\sigma \geq 0$ be defined by
\[
\sigma = \sum_{k=1}^n \Big [ \frac{1}{p_k} - \frac{1}{q_k}\Big].
\]
\begin{itemize}
\item[\textup(i)] There exists a number $N$ depending only on $n, p_1, p_2,\cdots, p_n$ and $q_1, q_2,\cdots, q_n$ such that
\begin{equation} \label{linear-decay.est}
\begin{split}
& \norm{e^{-\A t} \bP f}_{L_{q_1q_1\cdots q_n} (\bR^n)} \leq N t^{-\frac{\sigma}{2}} \norm{f}_{L_{p_1p_2\cdots p_n}(\bR^n)},\\
& \norm{D_{x} e^{-\A t} \bP f}_{L_{q_1q_1\cdots q_n} (\bR^n)} \leq N t^{-\frac{1}{2}(1 + \sigma)} \norm{f}_{L_{p_1p_2\cdots p_n}(\bR^n)},
\end{split}
\end{equation}
for every $f \in L_{p_1p_2\cdots p_n} (\bR^n)^n$.
\item[\textup{(ii)}] For each $f \in L_{p_1p_2\cdots p_n} (\bR^n)^n$, the following assertions hold
\begin{equation} \label{continuity}
\begin{split}
& \lim_{t\rightarrow 0^+} t^{\frac{\sigma}{2}}\norm{e^{-\A t} \bP f}_{L_{q_1q_1\cdots q_n} (\bR^n)} =0 \quad \text{if} \quad \sigma >0 \quad \text{and also} \\
& \lim_{t\rightarrow 0^+}\norm{[e^{-\A t} \bP f] -\bP f}_{L_{p_1p_1\cdots p_n} (\bR^n)} =0, \quad \text{and} \\
& \lim_{t\rightarrow 0^+}  t^{-\frac{1}{2}(1 + \sigma)}  \norm{D_{x} e^{-\A t} \bP f}_{L_{q_1q_1\cdots q_n} (\bR^n)} =0.
\end{split}
\end{equation}
\end{itemize}
\end{lemma}
\begin{proof} We begin with the proof of (i). As $\A = - \bP \Delta  = - \Delta \bP$, we see that $\A = -\Delta$ when acting on the class of divergenge free vector fields. Therefore, $e^{-\A t} \bP  = e^{\Delta t}\bP $. Then, by using the decay estimate for the heat equation in mixed norm developed in Theorem \ref{heat-decay}, we see that
\[
\norm{e^{-\A t} \bP f}_{L_{q_1q_1\cdots q_n} (\bR^n)} \leq N t^{-\frac{\sigma}{2}} \norm{\bP(f)}_{L_{p_1p_2\cdots p_n}(\bR^n)}.
\]
On the other hand,  from Corollary \ref{H-L-projection}, we see that the Helmholtz-Leray projection 
$$\bP: L_{p_1p_2\cdots p_n}(\bR^n)^n \rightarrow L_{p_1p_2\cdots p_n}(\bR^n)^n$$ 
is bounded. From this and the last estimate, we obtain the first estimate in \eqref{linear-decay.est}. The second estimate in \eqref{linear-decay.est} can be proved in the same way. 

Next, we prove (ii). We assume that $\sigma >0$ and we will prove the first assertion in \eqref{continuity}.  We may assume that $f$ is bounded and compactly supported if needed. Let $\epsilon >0$.  Then, by using approximation, we can find $g \in L_{p_1p_2\cdots p_n}(\bR^n)^n \cap L_{q_1q_2\cdots q_n}(\bR^n)^n$ such that
\[
\norm{f -g}_{L_{p_1p_2\cdots p_n}(\bR^n)} \leq \frac{\epsilon}{2N}
\]
where $N >0$ is defined in (i). Now, using the first assertion in (i), we see that
\[
t^{\frac{\sigma}{2}}\norm{e^{-\A t} \bP (f -g)}_{L_{q_1q_2\cdots q_n}(\bR^n)} \leq N \norm{f-g}_{L_{p_1p_2\cdots p_n}(\bR^n)} \leq \frac{\epsilon}{2}.
\]
On the other hand, using the first assertion in (i) again, we also obtain
\[
t^{\frac{\sigma}{2}}\norm{e^{-\A t} \bP g}_{L_{q_1q_2\cdots q_n}(\bR^n)} \leq N t^{\frac{\sigma}{2}} \norm{g}_{L_{q_1q_2\cdots q_n}(\bR^n)} \rightarrow 0 \quad \text{as} \quad t \rightarrow 0^+.
\]
Now, combine the last two estimates, we infer that there is small number $\delta_0 = \delta_0 (\epsilon) >0$ such that 
\[
 t^{\frac{\sigma}{2}}\norm{e^{-\A t} \bP f}_{L_{q_1q_1\cdots q_n} (\bR^n)} \leq \epsilon, \quad \forall \
t \in (0, \delta_0). \]
This implies that
\[
\lim_{t\rightarrow 0^+} t^{\frac{\sigma}{2}}\norm{e^{-\A t} \bP f}_{L_{q_1q_1\cdots q_n} (\bR^n)} =0,
\]
and the first assertion in \eqref{continuity} is proved.  Observe also that the last assertion in (ii) can be done in a similar way. Meanwhile, the second assertion of \eqref{continuity} is due to the continuity of the heat semi-group  in Lemma \ref{cont-lemma} and the continuity of the Helmholtz-Leray in the mixed norm $L_{p_1p_1\cdots p_n}(\bR^n)^n$ as from Corollary \ref{H-L-projection}. The proof of the lemma is therefore completed.
\end{proof}
Our next lemma gives some important estimates in mixed norm for the bilinear term $\G (u,v)$ defined in \eqref{G.def}.
\begin{lemma} \label{nonlinear-est} Let $p_k \in (1, \infty)$ and $\alpha_k, \beta_k, \gamma_k \in (0,1]$
be given numbers satisfying
\[
\gamma_k \leq \alpha_k + \beta_k < p_k, \quad k = 1,2,\cdots, n.
\]
Let
\[
\alpha = \sum_{k=1}^n \frac{\alpha_k}{p_k}, \quad  \beta = \sum_{k=1}^n \frac{\beta_k}{p_k}, \quad \text{and} \quad \gamma = \sum_{k=1}^n \frac{\gamma_k}{p_k}.
\]
Then,
\[
\begin{split}
&\norm{\G (u,v)(t)}_{L_{\frac{p_1}{\gamma_1} \frac{p_2}{\gamma_2} \cdots \frac{p_n}{\gamma_n}}(\bR^n)} \\
&  \qquad \leq N \int_0^t (t-s)^{-\frac{\alpha +\beta - \gamma}{2}} \norm{u(s)}_{L_{\frac{p_1}{\alpha_2} \frac{p_2}{\alpha_2} \cdots \frac{p_n}{\alpha_n}}(\bR^n)}\norm{D_x v(s)}_{L_{\frac{p_1}{\beta_1} \frac{p_2}{\beta_2} \cdots \frac{p_n}{\beta_n}}(\bR^n)} ds, \\
&\norm{D_x \G (u,v)(t)}_{L_{\frac{p_1}{\gamma_1} \frac{p_2}{\gamma_2} \cdots \frac{p_n}{\gamma_n}}(\bR^n)} \\
& \qquad \leq N \int_0^t (t-s)^{-\frac{1+ \alpha +\beta - \gamma}{2}} \norm{u(s)}_{L_{\frac{p_1}{\alpha_1} \frac{p_2}{\alpha_2} \cdots \frac{p_n}{\alpha_n}}(\bR^n)}\norm{D_x v(s)}_{L_{\frac{p_1}{\beta_1} \frac{p_2}{\beta_2} \cdots \frac{p_n}{\beta_n}}(\bR^n)} ds,
\end{split}
\]
where $N$ is a positive number depending only on $n$, $p_k, \alpha_l, \beta_k, \alpha_k$ for $k =1,2,\cdots, n$.
\end{lemma}
\begin{proof} We only prove the first assertion in the lemma as the proof of the second one can be done similarly. By applying the first estimate in \eqref{linear-decay.est}, we see that
\[
\norm{\G (u,v)(t)}_{L_{\frac{p_1}{\gamma_1} \frac{p_2}{\gamma_2} \cdots \frac{p_n}{\gamma_n}}(\bR^n)} \leq N \int_0^t (t-s)^{-\frac{\alpha +\beta - \gamma}{2}} \norm{F(u(s), v(s))}_{L_{\frac{p_1}{\alpha_1+ \beta_1} \frac{p_2}{\alpha_2+ \beta_2} \cdots \frac{p_n}{\alpha_n+\beta_n}}(\bR^n)} ds,
\]
where the bilinear function $F$ is defined in \eqref{F.def}.  From this and the boundedness of the Helmholtz-Leray projection $\bP$ as stated in Corollary \ref{H-L-projection}, we see that
\[
\norm{\G (u, v)(t)}_{L_{\frac{p_1}{\gamma_1} \frac{p_2}{\gamma_2} \cdots \frac{p_n}{\gamma_n}}(\bR^n)} \leq N \int_0^t (t-s)^{-\frac{\alpha +\beta - \gamma}{2}} \norm{(u\cdot\nabla) v}_{L_{\frac{p_1}{\alpha_1+ \beta_1} \frac{p_2}{\alpha_2+ \beta_2} \cdots \frac{p_n}{\alpha_n+\beta_n}}(\bR^n)} ds.
\]
Then, as
\[
\frac{\alpha_k + \beta_k}{p_k} = \frac{\alpha_k}{p_k} + \frac{\beta_k}{p_k}, \quad \text{for all} \  k = 1,2,\cdots, n
\]
we can repeatedly apply the H\"{o}lder's inequality for each integration with respect to each variable $x_k$ to find that
\[
\norm{(u\cdot\nabla) v}_{L_{\frac{p_1}{\alpha_1+ \beta_1} \frac{p_2}{\alpha_2+ \beta_2} \cdots \frac{p_n}{\alpha_n+\beta_n}}(\bR^n)} \leq \norm{u(s)}_{L_{\frac{p_1}{\alpha_1} \frac{p_2}{\alpha_2} \cdots \frac{p_n}{\alpha_n}}(\bR^n)}\norm{D_x v(s)}_{L_{\frac{p_1}{\beta_1} \frac{p_2}{\beta_2} \cdots \frac{p_n}{\beta_n}}(\bR^n)}.
\]
The desired estimate then follows and the proof is complete.
\end{proof}
To prove Theorem \ref{main-thrm}, our goal is to show that the abstract equation \eqref{u-abstract.eqn} has unique fixed points in suitable spaces. For this purpose, let us recall the following abstract lemma which is useful in the study of initial value problem for Navier-Stokes equations, see \cite[Lemma 3.1]{Phan-Phuc} and also \cite{Meyer}.
\begin{lemma}\label{abs-lemma} Let $X$ be a Banach space with norm $\norm{\cdot}_{X}$. Let $\G: X \times X \rightarrow X$ be a bilinear map such that there is $N_0 >0$ so that
\[
\norm{\G(u,v)}_X \leq N_0 \norm{u}_X \norm{v}_X, \quad \forall \ u ,\ v \in X.
\]
Then, for every $u_0 \in X$ with $4N_0 \norm{u_0}_X < 1$, the equation
\[
u = u_0 + \G(u,u)
\]
has unique solution $u \in X$ with
\[
\norm{u}_X \leq 2\norm{u_0}_X.
\]
\end{lemma}

We are now ready to prove Theorem \ref{main-thrm}.
\begin{proof}[Proof of Theorem \ref{main-thrm}]  Let $p = (p_1, p_2, \cdots, p_n), q = (q_1, q_2,\cdots q_n)$ with $p_k \in (2,\infty), q_k \in [p_k, \infty)$ for $k = 1, 2,\cdots, n$.  Assume that \eqref{p-cond}  and \eqref{q-cond} hold.  Let $a_0 \in L_{p_1p_2\cdots p_n}(\bR^n)^n$ with $\nabla \cdot a_0 =0$ and recall that
\begin{equation} \label{delta-cond}
 \delta = \frac{1}{q_1} + \frac{1}{q_2} + \cdots + \frac{1}{q_n}  \in (0, 1).
\end{equation}

We now prove (i). Recall the definitions of $\cX_{p, q, \infty}$ and $\cY_{p, q, \infty}$ in \eqref{X-spc.def} and \eqref{Y-spc.def}. We plan to prove the existence of solution $u \in \cX_{p, q, \infty}$ of \eqref{u-abstract.eqn}, and then prove that the solution $u \in \cY_{p, q, \infty}$. Our goal is to apply Lemma \ref{abs-lemma} to obtain the existence and uniqueness of solution of \eqref{u-abstract.eqn} in $\cX_{p, q, \infty}$.  To this end, we begin with the proof that $u_0 \in \cX_{p, q, \infty}$.  From (i) of Lemma \ref{linear-est} and the definition of $u_0$ in \eqref{G.def}, we have
\[
\begin{split}
& \norm{u_0(t)}_{L_{q_1q_2\cdots q_n}(\bR^n)} \leq N_1 t^{-\frac{1-\delta}{2}} \norm{a_0}_{L_{p_1p_2\cdots p_n}(\bR^n)} \quad \text{and} \\
& \norm{D_x u_0(t)}_{L_{p_1p_2\cdots p_n}(\bR^n)} \leq N_1 t^{-\frac{1}{2}} \norm{a_0}_{L_{p_1p_2\cdots p_n}(\bR^n)}, \quad \forall \ t >0,
\end{split}
\]
where $N_1>0$ is a universal constant depending only on $n, p$ and $q$. Moreover, it follows from  (ii) of Lemma \ref{linear-est} that $t^{(1-\delta)/2} e^{-\A t}\bP$ is uniformly bounded from $L_{p_1p_2\cdots p_n}(\bR^n)^n$ to $PL_{q_1q_2\cdots q_n}(\bR^n)$ and tends to zero as $t\rightarrow 0^+$, we see that $t^{(1-\delta)/2} u_0$ vanishes as $t =0$. Similarly, as $t^{1/2} D_x e^{-\A t}\bP$ is uniformly bounded from $L_{p_1p_2\cdots p_n}(\bR^n)^n$ to $PL_{p_1p_2\cdots p_n}(\bR^n)^n$ and tends to zero as $t \rightarrow 0^+$, we also have $t^{1/2} D_x u_0$ equals to zero as $t \rightarrow 0^+$. In conclusion, we have shown that $u_0 \in \cX_{p, q,\infty}$ and
\begin{equation} \label{u-0-norm}
\norm{u_0}_{\cX_{p, q, \infty}} \leq N_1 \norm{a_0}_{L_{p_1p_2\cdots p_n}(\bR^n)}.
\end{equation}
It now remains to prove that the bilinear form $\G: \cX_{p, q, \infty} \times \cX_{p, q, \infty} \rightarrow \cX_{p, q, \infty}$ is bounded. By \eqref{delta-cond} and \eqref{p-cond}, we apply the first assertion in Lemma \ref{nonlinear-est} with $\beta_k =1$ and $\gamma_k = \alpha_k = \frac{p_k}{q_k} \in (0, 1]$ to find that
 \[
 \begin{split}
&  \norm{\G(u,v)(t)}_{L_{q_1q_2\cdots q_n}(\bR^n)} \\
& \leq N \int_{0}^t (t-s)^{-\frac{1}{2}} \norm{u(s)}_{L_{q_1q_2\cdots q_n}(\bR^n)} \norm{D_x v(s)}_{L_{p_1p_2\cdots p_n}(\bR^n)} ds \\
 & \leq N \norm{u}_{\cX_{p, q, \infty}} \norm{v}_{\cX_{p, q, \infty}} \int_{0}^t (t-s)^{-\frac{1}{2}} s^{-1 + \frac{\delta}{2}} ds.
\end{split}
\]
To control the integration in the last estimate, we split it into two time intervals $(0, t/2)$ and $(t/2, t)$. We then obtain
\[
\begin{split}
& \norm{\G(u,v)(t)}_{L_{q_1q_2\cdots q_n}(\bR^n)}\\
& \leq N \norm{u}_{\cX_{p, q, \infty}} \norm{v}_{\cX_{p, q, \infty}} \left[ \int_{0}^{t/2} (t-s)^{-\frac{1}{2}} s^{-1 + \frac{\delta}{2}} ds +  \int_{t/2}^{t/2} (t-s)^{-\frac{1}{2}} s^{-1 + \frac{\delta}{2}} ds \right]\\
& \leq N \norm{u}_{\cX_{p, q, \infty}} \norm{v}_{\cX_{p, q, \infty}} \left[t^{-\frac{1}{2}}  \int_{0}^{t/2} s^{-1 + \frac{\delta}{2}} ds +  t^{-1 + \frac{\delta}{2}} \int_{t/2}^{t/2} (t-s)^{-\frac{1}{2}}  ds \right]\\
 & \leq N t^{-\frac{1-\delta}{2}} \norm{u}_{\cX_{p, q, \infty}} \norm{v}_{\cX_{p, q, \infty}}.
\end{split}
 \]
 Similarly,  By using \eqref{delta-cond} and \eqref{p-cond}, and applying the second assertion in Lemma \ref{nonlinear-est} with $\gamma_k =1$, $\beta_k =1$ and $\alpha_k = \frac{p_k}{q_k} \in (0, 1]$, we also have
 \begin{equation} \label{dx-G-u-v.est}
 \begin{split}
 & \norm{D_x\G(u,v)(t)}_{L_{p_1p_2\cdots p_n}(\bR^n)} \\
 &  \leq N \int_{0}^t (t-s)^{-\frac{1+\delta}{2}} \norm{u(s)}_{L_{q_1q_2\cdots q_n}(\bR^n)} \norm{D_x u(s)}_{L_{p_1p_2\cdots p_n}(\bR^n)} ds\\
 & \leq  N \norm{u}_{\cX_{p, q, \infty}} \norm{v}_{\cX_{p, q, \infty}} \int_{0}^t (t-s)^{-\frac{1+\delta}{2}} s^{-1 +\frac{\delta}{2}} ds \\
& =  N \norm{u}_{\cX_{p, q, \infty}} \norm{v}_{\cX_{p, q, \infty}} \left[ \int_{0}^{t/2}  (t-s)^{-\frac{1+\delta}{2}} s^{-1 +\frac{\delta}{2}} ds + \int_{t/2}^t (t-s)^{-\frac{1+\delta}{2}} s^{-1 +\frac{\delta}{2}} ds \right] \\
& =  N \norm{u}_{\cX_{p, q, \infty}} \norm{v}_{\cX_{p, q, \infty}} \left[t^{-\frac{1+\delta}{2}} \int_{0}^{t/2} s^{-1 +\frac{\delta}{2}} ds + t^{-1 +\frac{\delta}{2}} \int_{t/2}^t (t-s)^{-\frac{1+\delta}{2}} ds \right]\\
 & \leq N  t^{-1/2} \norm{u}_{\cX_{p, q, \infty}} \norm{v}_{\cX_{p, q, \infty}}.
 \end{split}
 \end{equation}
 From the last two estimates and the definition of $\G(u,v)$ and Lemma \ref{linear-est}, it follows that $t^{(1-\delta)/2}\G(u,v): [0,\infty) \rightarrow PL_{q_1q_2\cdots q_n}(\bR^n)$ is continuous and vanishes at $t =0$.  Similarly, we can also prove that  $t^{1/2}D_x\G(u,v) : [0, \infty) \rightarrow  PL_{p_1p_2\cdots p_n}(\bR^n)$ is continuous and vanishes at $t =0$. Therefore, we conclude that $G(u,v) \in \cX_{p, q, \infty}$ and
\begin{equation} \label{G-bounded-inft}
 \norm{\G(u,v)}_{\cX_{p, q, \infty}} \leq N_2 \norm{u}_{\cX_{p, q, \infty}} \norm{v}_{\cX_{p, q, \infty}}, \quad \forall \ u ,v \in \cX_{p, q, \infty},
 \end{equation}
where $N_2$ is a constant depending only on $n, p$ and $q$.  In other words, the bilinear form  $\G: \cX_{p, q, \infty} \times \cX_{p, q, \infty}  \rightarrow \cX_{p, q, \infty}$ is  bounded. 

Next, let us choose $\lambda_0 >0$ and sufficiently small so that
\begin{equation} \label{lambda-zero}
4N_1 N_2 \lambda_0 <1,
\end{equation}
where $N_1$ is defined in \eqref{u-0-norm},  and $N_2$ is defined in \eqref{G-bounded-inft}. Note that  both of these numbers depend only on $p, q$ and $n$.  Now,  if $\norm{a_0}_{L_{p_1p_2\cdots p_n}(\bR^n)} \leq \lambda_0$, then it follows from \eqref{u-0-norm} that
\[
4N_2 \norm{u_0}_{\cX_{p ,q, \infty}} \leq 4N_1N_2 \norm{a_0}_{L_{p_1p_2\cdots p_n}(\bR^n)} \leq 4N_1N_2 \lambda_0 <1.
\]
From this and by applying Lemma \ref{abs-lemma}, we can find a unique solution $u \in \cX_{p, q, \infty}$ of the equation \eqref{u-abstract.eqn} such that
\begin{equation} \label{u-est-infty}
\norm{u}_{\cX_{p ,q, \infty}} \leq 2\norm{u_0}_{X_\infty} \leq 2N_1 \norm{a_0}_{L_{p_1p_2\cdots p_n}(\bR^n)}.
\end{equation}
Now, to complete the proof (i), we need to show that $u \in \cY_{p, q, \infty}$. We recall that the definition of $\cY_{p, q, \infty}$ is given in \eqref{Y-spc.def}.  Since
\[
u(t) = u_0(t) + \G(u, u) (t),
\]
we have
\begin{equation} \label{u-Y-all}
\begin{split}
& \norm{u(t)}_{L_{p_1p_2\cdots p_n}(\bR^n)} \leq \norm{u_0(t)}_{L_{p_1p_2\cdots p_n}(\bR^n)} + \norm{\G(u, u) (t)}_{L_{p_1p_2\cdots p_n}(\bR^n)}, \quad \text{and}\\
& \norm{D_x u(t)}_{L_{p_1p_2\cdots p_n}(\bR^n)} \leq \norm{D_x u_0(t)}_{L_{p_1p_2\cdots p_n}(\bR^n)} + \norm{D_x \G(u, u) (t)}_{L_{p_1p_2\cdots p_n}(\bR^n)}.
\end{split}
\end{equation}
Then, by applying Lemma \ref{linear-est}, we see that
\begin{equation} \label{u-zero-Y}
\begin{split}
& \norm{u_0(t)}_{L_{p_1p_2\cdots p_n}(\bR^n)} \leq N \norm{a_0}_{L_{p_1p_2\cdots p_n}(\bR^n)}, \quad \text{and} \\
&  \norm{D_x u_0(t)}_{L_{p_1p_2\cdots p_n}(\bR^n)} \leq N t^{-1/2} \norm{a_0}_{L_{p_1p_2\cdots p_n}(\bR^n)}.
\end{split}
\end{equation}
On the other hand, by \eqref{delta-cond} and \eqref{p-cond}, we can apply the first assertion in Lemma \ref{nonlinear-est} with $\gamma_k =1$, $\alpha_k = \frac{p_k}{q_k} \in (0, 1]$ and $\beta_k =1$ to infer  that
\begin{equation} \label{G-u-u-p.est}
\begin{split}
& \norm{\G(u,u)(t)}_{L_{p_1p_2\cdots p_n}(\bR^n)} \\
&  \leq N \int_{0}^t (t-s)^{-\frac{\delta}{2}}\norm{u(s)}_{L_{q_1q_2\cdots q_n}(\bR^n)} \norm{D_x u(s)}_{L_{p_1p_2\cdots p_n}(\bR^n)} ds\\
& \leq  N \norm{u}_{\cX_{p, q, \infty}}^2 \int_{0}^t (t-s)^{-\frac{\delta}{2}} s^{-(1-\frac{\delta}{2})} ds\\
& = N \norm{u}_{\cX_{p, q, \infty}}^2  \left [ \int_{0}^{t/2}  (t-s)^{-\frac{\delta}{2}} s^{-(1-\frac{\delta}{2})} ds + \int_{t/2}^{t}  (t-s)^{-\frac{\delta}{2}} s^{-(1-\frac{\delta}{2})} ds \right]\\
& = N \norm{u}_{\cX_{p, q, \infty}}^2  \left [ t^{-\frac{\delta}{2}} \int_{0}^{t/2}  s^{-(1-\frac{\delta}{2})} ds + t^{-(1-\frac{\delta}{2})} \int_{t/2}^{t}  (t-s)^{-\frac{\delta}{2}}  ds \right]\\
& \leq N \norm{a_0}^2_{L_{p_1p_2\cdots p_n}(\bR^n)},
\end{split}
\end{equation}
where in the last estimate, we used \eqref{u-est-infty}. Also, by  \eqref{dx-G-u-v.est} and  \eqref{u-est-infty}, it follows that
\begin{equation} \label{Dx-u-u.est}
\norm{D_x \G(u,u)}_{L_{p_1p_2\cdots p_n}(\bR^n)}  \leq N t^{-1/2}\norm{u}_{\cX_{p, q, \infty}}^2 \leq   N t^{-1/2}\norm{a_0}^2_{L_{p_1p_2\cdots p_n}(\bR^n)}.
\end{equation}
Then,  from the estimates \eqref{u-Y-all}, \eqref{u-zero-Y}, \eqref{G-u-u-p.est}, \eqref{Dx-u-u.est} and the fact that $\norm{a_0}_{L_{p_1p_2\cdots p_n}(\bR^n)}$ is sufficiently small that, we see that 
\[
\norm{u}_{\cY_{p, q, \infty}} \leq N_0\norm{a_0}_{L_{p_1p_2\cdots p_n}(\bR^n)}.
\]
The proof of (i) is therefore complete.

Now, we turn to prove (ii).  As in the proof of \eqref{u-0-norm}, we see that $u_0 \in \cX_{p, q, \infty}$.  From the definition of the norm of the space $ \cX_{p, q, \infty}$ in \eqref{X-spc.def}, the continuity and the vanishes of $t^{(1-\delta)/2}u_0$ and of $t^{1/2}D_x u_0$ at $t =0$,  we can choose a sufficiently small number $T_0>0$ depending on $n, p, q$ and $a_0$ so that
\[
\norm{u_0}_{\cX_{p, q, T_0}} \leq \lambda_0,
\]
where $\lambda_0$ is defined as in \eqref{lambda-zero}.  Moreover, by following the proof of \eqref{G-bounded-inft}, we can also see that the bilinear form $\G: \cX_{p, q, T_0} \times \cX_{p, q, T_0} \rightarrow \cX_{p, q, T_0}$ is bounded  with
\[
\norm{\G(u,v)}_{\cX_{p, q, T_0}} \leq N_2\norm{u}_{\cX_{p, q, T_0}} \norm{v}_{\cX_{p, q, T_0}}, \quad \forall \ u, v \in \cX_{p, q, T_0}.
\]
Then, applying Lemma \ref{abs-lemma} again, we can find a unique local time solution $u \in \cX_{p, q, T_0}$ of  \eqref{u-abstract.eqn} satisfying
\[
\norm{u}_{\cX_{p, q, T_0}} \leq 2N_1 \norm{a_0}_{L_{p_1p_2\cdots p_n}(\bR^n)}.
\]
Now, we only need to prove that the solution $u$ that we found is indeed in $\cY_{p, q, T_0}$.  However, this can be done exactly the same as in the proof  that $u \in \cY_{p, q, \infty}$ in (i), and we skip it. The proof of the theorem is then complete.
\end{proof}
\begin{remark} The pressure $P$ in \eqref{NS.eqn} can be solved from the solution $u = (u_1, u_2,\cdots, u_n)$ as
\[
P = \sum_{i,j =1}^n \mathscr{R}_i\mathscr{R}_j(u_i u_j),
\]
where $\mathscr{R}_j$ is the $i^{th}$ Riesz transform, which is defined in Theorem \ref{Riesz}. Since $p_k >2$ for all $k =1,2,\cdots, n$, we can apply Theorem \ref{Riesz} to obtain
\[
\norm{P(\cdot, t)}_{L_{\frac{p_1}{2} \frac{p_2}{2} \cdots \frac{p_n}{2} }(\bR^n)} \leq N \norm{u(\cdot, t)}_{L_{p_1p_2\cdots p_n}(\bR^n)}.
\]
\end{remark}

\noindent \textbf{Acknowledgement.}  The author would like to thanks professor Lorenzo Brandolese (Institut Camille Jordan, Universit\'{e} Lyon 1) and professor Nam Le (Indiana University) for their valuable comments.

\ \\ \\

\begin{thebibliography}{m}
\bibitem{Chemin}  H. Bahouri,  J.-Y. Chemin,  R. Danchin,  {\it Fourier analysis and nonlinear partial differential equations. } Grundlehren der Mathematischen Wissenschaften, 343. Springer, Heidelberg, 2011.

\bibitem{Bourgain-Pavlovic} J. Bourgain, N. Pavlovi\'{c},  {\it Ill-posedness of the Navier-Stokes equations in a critical space in 3D}, J. Funct. Anal. 255 (2008), 2233-2247.

\bibitem{Brandolese}  L. Brandolese,  F. Vigneron,  {\it New asymptotic profiles of nonstationary solutions of the Navier-Stokes system}. J. Math. Pures Appl. (9) 88 (2007), no. 1, 64-86.

\bibitem{Cannone}  M. Cannone, { \it Ondelettes, paraproduits et Navier-Stokes}.  Diderot Editeur, Paris, 1995.

\bibitem{Cannone-2} M. Cannone,  {\it A generalization of a theorem by Kato on Navier-Stokes equations}, Rev. Mat. Iberoam. 13 (1997), 515-541.

\bibitem{Cannone-Planchon}  M. Cannone,  F. Planchon, {\it On the non-stationary Navier-Stokes equations with an external force}. Adv. Differential Equations 4 (1999), no. 5, 697-730.

\bibitem{David}  D. V. Cruz-Uribe, J. M. Martell, and C. P\'{e}rez.  {\it Weights, extrapolation and the theory of Rubio de Francia,} volume 215 of Operator Theory: Advances and Applications. Birkh\"{a}user/Springer Basel AG, Basel, 2011.

\bibitem{Dong-Kim}  H. Dong,  D. Kim, {\it On $L_p$-estimates for elliptic and parabolic equations with $A_p$ weights}. Trans. Amer. Math. Soc. 370 (2018), no. 7, 5081-5130.

\bibitem{Dong-Krylov} H. Dong, N.V. Krylov,  {\it Fully nonlinear elliptic and parabolic equations in weighted and mixed-norm Sobolev spaces},  arXiv:1806.00077.

\bibitem{Dong-Phan} H. Dong, T. Phan, {\it Mixed norm $L_p$-estimates for non-stationary Stokes systems with singular VMO coefficients and applications},  arXiv:1805.04143.

\bibitem {Fu-Ka-1} T. Kato, H. Fujita, {\it On the nonstationary Navier-Stokes system}. Rend. Sem. Mat. Univ. Padova 32 1962 243-260.

\bibitem{Fu-Ka} H. Fujita,  T. Kato, {\it  On the Navier-Stokes initial value problem. I}. 
Arch. Rational Mech. Anal. 16, 1964, 269-315.

\bibitem{RF}  J. Garc\'{i}a-Cuerva,  J. L. Rubio de Francia,  {\it Weighted norm inequalities and related topics}. North-Holland Mathematics Studies, 116. Notas de Matem\'{a}tica, 104. North-Holland Publishing Co., Amsterdam, 1985.

\bibitem{Giga-1} Y. Giga, T. Miyakawa, {\it Solutions in $L^r$ of the Navier-Stokes initial value problem}. Arch. Rational Mech. Anal. 89 (1985), no. 3, 267-281.

\bibitem{Giga-M} Y. Giga and T. Miyakawa,  {\it Navier-Stokes flow in $\bR^3$ with measures as initial vorticity and Morrey spaces}, Comm. Partial Differential Equations 14 (1989), 577-618.

\bibitem{Giga} Y. Giga, {\it Solutions for semilinear parabolic equations in $L^p$ and regularity of weak solutions of the Navier-Stokes system}. J. Differential Equations 62 (1986), no. 2, 186-212.

\bibitem{Kato} T. Kato,  {\it Strong $L^p$-solutions of the Navier-Stokes equation in $\bR^m$, with applications to weak solutions}. Math. Z. 187 (1984), no. 4, 471-480.

\bibitem{Kato-1} T. Kato, {\it Strong solutions of the Navier-Stokes equation in Morrey spaces}, Bol. Soc. Brasil. Mat. (N.S.) 22 (1992), 127-155.

\bibitem{Koch-Tataru} H. Koch, D. Tataru, {\it Well-posedness for the Navier-Stokes equations.} Adv. Math. 157 (2001), no. 1, 22-35.

\bibitem{K-Ya} H. Kozono, M. Yamazaki, {\it Semilinear heat equations and the Navier-Stokes equation with distributions in new function spaces as initial data}, Comm. Partial Differential Equations 19 (1994), 959-1014.

\bibitem{Kazono-Ya}  H. Kozono, M. Yamazaki,  {\it The stability of small stationary solutions in Morrey spaces of the Navier-Stokes equation}, Indiana Univ. Math. J. 44 (1995), 1307-1335.

\bibitem{Krylov-survey} N.V. Krylov, {\it Rubio de Francia extrapolation theorem and related topics in the theory of elliptic and parabolic equations. A survey},  arXiv:1901.00549.

\bibitem{Krylov}  N. V. Krylov,  {\it Parabolic equations with VMO coefficients in Sobolev spaces with mixed norms}. J. Funct. Anal. 250 (2007), no. 2, 521-558.

\bibitem{L-R} P. R.  Lemari\'{e}-Rieusset,  {\it 
The Navier-Stokes problem in the 21st century}. CRC Press, Boca Raton, FL, 2016.

\bibitem{Leray} J. Leray,  {\it \'{E}tude de diverses \'{e}quations int\'{e}grales non lin\'{e}aires et de quelques probl\'{e}mes que pose l'Hydrodynamique}, J. Math. Pures Appl. 9 (1933), 1-82.

\bibitem{Leray-1}  J. Leray, {\it Sur le mouvement d'un liquide visqueux emplissant l'espace}, Acta Math. 63 (1934), no. 1, 193-248.

\bibitem{Meyer} Y. Meyer, {\it Wavelets, Paraproducts and Navier-Stokes Equations}. Current Developments in Mathematics, 1996 (Cambridge, MA), Int. Press, Boston, MA, 1997, pp. 105-212.

\bibitem{F-Planchon} F. Planchon, {\it Global strong solutions in Sobolev or Lebesgue spaces to the incompressible Navier-Stokes equations in $\bR^3$}, Ann. Inst. Henri Poincare, Anal. Non Lineaire 13 (1996), 319-336.

\bibitem{Phan-Phuc}  T. V. Phan, N. C. Phuc, {\it Stationary Navier-Stokes equations with critically singular external forces: existence and stability results}. Adv. Math. 241 (2013), 137-161.

\bibitem{TP} T. Phan, {\it Liouville type theorems for 3D stationary Navier-Stokes equations in weighted mixed-norm Lebesgue spaces}, arXiv:1812.10135.

\bibitem{Taylor} M.E. Taylor,  {\it Analysis of Morrey spaces and applications to Navier-Stokes and other evolution equations}, Comm. Partial Differential Equations 17 (1992), 1407-1456.

\bibitem{Tsai} T.-P. Tsai,  {\it Lectures on Navier-Stokes equations}. Graduate Studies in Mathematics, 192. American Mathematical Society, Providence, RI, 2018.




\end{thebibliography}
 \end{document}